\documentclass[12pt]{article}
\usepackage{graphicx}
\usepackage{amsfonts}
\usepackage{enumerate}

\usepackage[all]{xy}

\def\squarebox#1{\hbox to #1{\hfill\vbox to #1{\vfill}}}
\newcommand{\qed}{\hspace*{\fill}
\vbox{\hrule\hbox{\vrule\squarebox{.667em}\vrule}\hrule}\smallskip}
\newtheorem{teorema}{Theorem}[section]
\newtheorem{lema}[teorema]{Lemma}
\newtheorem{corolario}[teorema]{Corollary}
\newtheorem{proposicao}[teorema]{Proposition}

\newenvironment{proof}{\noindent {\bf Proof:}}{\hfill $\qed $ \newline}

\newcommand{\R}{{\mathbb R}}

\newcommand{\Z}{{\mathbb Z}}

\newcommand{\F}{\mathbb{F}}

\newcommand{\tdcds}{{\mathbb T}}

\newcommand{\ad}{{\rm ad}}

\newcommand{\g}{\mathfrak{g}}

\newcommand{\n}{\mathfrak{n}}
\newcommand{\p}{\mathfrak{p}}
\renewcommand{\l}{\mathfrak{l}}
\renewcommand{\b}{\mathfrak{b}}

\newcommand{\T}{\Theta}

\def\prod#1{\left\langle{#1}\right\rangle}

\begin{document}

\title{Lyapunov, metric and flag spectra}
\author{Mauro Patr\~{a}o}
\maketitle

\begin{abstract}
We introduce the \emph{metric spectrum}, which measures the
exponential rate of approximation to an isolated invariant set of
points starting in its stable set, and relate it to the Lyapunov
spectrum. We determine the metric spectrum of each Morse component
of the finest Morse decomposition of a linear induced flow on a
generalized flag manifold.
\end{abstract}

\noindent \textit{AMS 2000 subject classification}: Primary:
37B35, 37D99, Secondary: 14M15.

\noindent \textit{Key words:} Lyapunov exponents, stable sets,
generalized flag manifolds.

\section{Introduction}

Let $\phi^t$ be a discrete or continuous-time flow defined on a
metric space $(X,d)$. If $M \subset X$ is a compact isolated
$\phi^t$-invariant set which is not a repeller, we introduce an
spectrum, called \emph{metric spectrum}, which measures the
exponential rate of approximation to $M$ of points starting in the
stable set of $M$, as illustrated by Figure \ref{figura1}.
\begin{figure}[h]
    \begin{center}
        \includegraphics[scale=1]{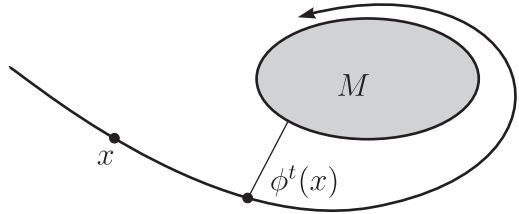}
    \end{center}
    \caption{\label{figura1}
    The approximation of $\phi^t(x)$ to $M$.}
\end{figure}

We show that this spectrum is an invariant of the conjugation
classes by bi-Lipschitz maps. In particular, when $X$ is a
manifold, the metric spectrum is independent of Riemannian
metrics. When the restriction of $\phi^t$ to an open neighborhood
of $M$ is conjugated to a linear flow $\Phi^t$, the metric
spectrum is given by the negative Lyapunov exponents of $\Phi^t$.
In particular, if $M$ is normally hyperbolic, this implies that
its metric spectrum is not empty.

Now let $g^t$ be a linear induced flow on a flag manifold $\F_\T$
of a connected noncompact real semi-simple Lie group $G$. Assuming
that $g^t$ is conformal (its unipotent Jordan component is
trivial), we determine explicitly the metric spectrum of each
Morse component of the finest Morse decomposition. This is done by
combining the linearization presented in \cite{pss-conley} with
some results on the Jordan decomposition of $g^t$ presented in
\cite{pfs-jordan}. This result generalizes the following simple
situation. Let $G = \mbox{Sl}(2,\mathbb{R})$, $\F_\T =
\mathbb{P}\mathbb{R}^2$ and
\[
g^t = \left(
\begin{array}{cc}
\mbox{e}^{\lambda t} & 0                      \\
0                    & \mbox{e}^{-\lambda t}
\end{array}
\right).
\]
We have that $[e_1]$ and $[e_2]$  are, respectively, the attractor
and repeller components of $g^t$ in $\mathbb{P}\mathbb{R}^2$. For
each $[v] \neq [e_2]$, we have that $g^t[v] \to [e_1]$, when $t
\to \infty$. The $\mbox{SO}(2,\mathbb{R})$-invariant distance $d$
in $\mathbb{P}\mathbb{R}^2$ is such that $d(g^t[v], [e_1]) =
\theta_t$, where $\theta_t$ is the angle illustrated in the Figure
\ref{figura2}.

\begin{figure}[h]
    \begin{center}
        \includegraphics[scale=1]{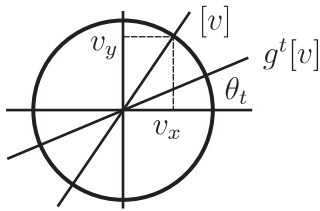}
    \end{center}
    \caption{\label{figura2}
    The approximation of $g^t[v]$ to $[e_1]$ in $\mathbb{P}\mathbb{R}^2$.}
\end{figure}

Since
\[
\lim_{t \to \infty} \frac{\mbox{tg}(\theta_t)}{\theta_t} = 1,
\]
it follows that
\[
\lim_{t \to \infty} \frac{1}{t} \log d(g^t[v], [e_1]) = -2\lambda,
\]
where we use that
\[
\mbox{tg}(\theta_t) = \frac{v_y}{v_x}\,\mbox{e}^{-2\lambda t}.
\]
For the attractor component of $g^t$ in the maximal flag manifold
$\F$, we do not need to assume that $g^t$ is conformal. When $g^t$
is arbitrary, the metric spectrum of its attractor component in
the maximal flag manifold will be called the \emph{flag spectrum}
of $g^t$.

\section{Preliminaries}\label{preliminaries}

For the theory of semi-simple Lie groups and their flag manifolds
we refer to Duistermat-Kolk-Varadarajan \cite{dkv}, Helgason
\cite{helgason}, Knapp \cite{knapp} and Warner \cite{w}. To set
notation let $G$ be a connected noncompact real semi-simple Lie
group with Lie algebra $\frak{g}$. Let ${\rm Ad}: G \to {\rm
Gl}({\frak g})$ be the adjoint representation of $G$. An element
$g \in G$ acts in the Lie algebra ${\frak g}$ by the adjoint
representation, so that for $X \in {\frak g}$ we write
\[
gX = {\rm Ad}(g)X.
\]
With this notation we have that
\[
g \exp(X) g^{-1} = \exp( gX )\qquad\mbox{and}\qquad \exp(X)Y =
{\rm e}^{{\rm ad}(X)}Y,
\]
for $g \in G$, $X, Y \in {\frak g}$.

Fix a Cartan involution $\theta $ of $ \frak{g}$, a maximal
abelian subspace $\frak{a}\subset \frak{s}$ and a Weyl chamber
$\frak{a}^{+}\subset \frak{a}$. Let $\frak{g}=\frak{k}\oplus
\frak{s}$ be the Cartan decomposition and $\prod{\cdot,\cdot}$ be
the Cartan inner product associated to $\theta$. We let $\Pi $ be
the set of roots of $\frak{a }$, $\Pi ^{+}$ the positive roots
corresponding to $\frak{a}^{+}$, $\Sigma $ the set of simple roots
in $\Pi ^{+}$.

For each $\Theta \subset \Sigma $, we denote by $\F_\T$ the
associated flag manifold, which is a homogeneous space of $G$. Let
$b_\T$ be the base point, $P_\T$ be its isotropy group and $\p_\T$
be its isotropy algebra. Each element $X \in {\frak g}$ of the Lie
algebra induces a differentiable flow in the flag manifold
${\mathbb F}_{\Theta }$ given by
\[
(t,x) \mapsto \exp(tX)x, \qquad t \in {\mathbb R},\, x \in
{\mathbb F}_{\Theta }.
\]
This flow is generated by the induced vector field which is
denoted by $X \cdot $ and its value at $x$ is denoted by $X \cdot
x$. An element $g \in G$ acts in a tangent vector $v \in T_x
{\mathbb F}_{\Theta }$ by its differential at $x$. We denote this
action by $g v = d_x g(v)$.  For induced vector fields we have
that
\[
g ( X \cdot x) = gX \cdot gx.
\]
For a fixed $x \in \F_\T$, the map $X \mapsto X \cdot x$ is a
linear map from $\g$ to $T_x\F_\T$ whose kernel is the subalgebra
of isotropy at $x$.

We have that the compact group $K = \exp(\frak{k})$ is transitive
in $\F_\T$, its isotropy subgroup at $b_\T$ is denoted by $K_\T$
and its elements acts in $\g$ by $\prod{\cdot,\cdot}$-isometries.
There exists a $K_\T$-invariant subalgebra $\n_\T$ which
complements $\p_\T$ in $\g$. Thus we can identify the tangent
space of $\F_\T$ at $b_\T$ with $\n_\T$ and $\prod{\cdot,\cdot}$
determines a $K$-invariant Riemannian metric in $\F_\T$ given by
\[
|X \cdot x| = |Y|,
\]
where $X \cdot x = k(Y \cdot b_\T)$ with $k \in K$. The Weyl group
$W=M^{*}/M$ acts on $\frak{a}$ by isometries, where $M^{*}$ and
$M$ are, respectively, the normalizer and the centralizer of
$\frak{a}$ in $K$.

We denote by $g^t$, $t \in \tdcds = \R$ or $\Z$, the right
invariant continuous-time flow generated by $X \in \g$ or the
discrete-time flow generated by $g \in G$. More precisely, when
$\tdcds = \R$, we have that $g^t = \exp(tX)$ and, when $\tdcds =
\Z$, we have that $g^t$ is the $t$-iterate of $g$. The flow
induced by $g^t$ on $\F_\T$ is called a linear induced flow. Let
$g^t = e^t h^t u^t$ be the multiplicative Jordan decomposition of
$g^t$, a commutative composition of linear induced flows. The
elliptic, hyperbolic and unipotent components of $g^t$ are,
respectively, the flow $e^t$, the flow $h^t = \exp(tH)$ and the
flow $u^t = \exp(tN)$, where $H$ is called the hyperbolic type of
the flow $g^t$ and $\ad(N)$ is a nilpotent operator. The
hyperbolic component is gradient with respect to a given
Riemannian metric on ${\mathbb F}_{\Theta }$ (see \cite{dkv},
Section 3). The connected components of fixed point set of this
flow are given by
\[
\mathrm{fix}_{\Theta }(H,w) = K_{H} wb_{\Theta },
\]
where $K_{H}$ is the centralizer of $H$ in $K$. The stable and
unstable sets of ${\rm fix}_\T(H,w)$ with respect to $h^t$ are
given, respectively, by
\[
\mathrm{st}_{\Theta }(H,w)=N_{H}^{-} \mathrm{fix}_{\Theta }(H,w)
\qquad {\rm and} \qquad \mathrm{un}_{\Theta }(H,w) = N_{H}^{+}
\mathrm{fix}_{\Theta }(H,w)
\]
where $N_H^{\pm} = \exp(\n_H^\pm)$,
\[
\frak{n}^-_H = \sum\{ \frak{g}_\alpha:\, \alpha(H) < 0 \} \qquad
{\rm and} \qquad \frak{n}^+_H = \sum\{ \frak{g}_\alpha:\,
\alpha(H) > 0 \}.
\]
For an arbitrary linear induced flow on $\F_\T$, we have the
following similar description for its finest Morse decomposition
(see Proposition 5.1 of \cite{pfs-jordan}).

\begin{proposicao}\label{propmorseflag}
Let $g^t$ be a linear induced flow on $\F_\T$. The set
$$
\{{\rm fix}_\T(H,w): w \in W \}
$$
is the finest Morse decomposition for $g^t$. Furthermore, the
stable and unstable sets of ${\rm fix}_\T(H,w)$ with respect to
$g^t$ are given, respectively, by
$$
{\rm st}_\T(H,w) \qquad \emph{and} \qquad {\rm un}_\T(H,w).
$$
\end{proposicao}

Now we define the subspace
\begin{equation}\label{eqdefl}
{\frak l}^\pm_{wb_\T} = \frak{n}^\pm_H \cap w{\frak n}^-_\T
\end{equation}
and the family of subspaces $\frak{l}^\pm_{x}\subset
\frak{n}^\pm_H$ given by
\[
{\frak l}^\pm_x = k {\frak l}^\pm_{wb_\T},
\]
for $x = k w b_\Theta$, where $k \in K_H$. By Proposition 3.1 of
\cite{pss-conley}, the families $\frak{l}^\pm_{x}$ are well
defined, each $\l_x$ is $h^t$-invariant and, for $k \in K_H$, we
have
\[
k \l^\pm_x = \l^\pm_{kx}.
\]
Furthermore we have that the map
\[
X \in \frak{l}^\pm_{x}\mapsto X \cdot x \in \frak{l}^\pm_{x}\cdot
x
\]
is a linear isomorphism, for each $x \in \mathrm{fix}_{\Theta
}(H,w)$.

By Propositions 3.2 and 3.3 of \cite{pss-conley}, we have that
\[
V_{\Theta }^\pm(H,w) = \bigcup \{ \frak{l}^\pm_{x}\cdot x:\, x \in
\mathrm{fix}_{\Theta }(H,w) \}
\]
are differentiable subbundles of the tangent bundle of $\F_\T$
over $\mathrm{fix}_{\Theta }(H,w)$ such that its Whitney sum
\begin{equation}\label{eqwhitneysum}
V_{\Theta }\left( H,w\right) = V_{\Theta }^{+}(H,w)\oplus
V_{\Theta }^{-}(H,w)
\end{equation}
is the normal bundle of $\mathrm{fix}_{\Theta }(H,w)$ and whose
fiber at $x \in \mathrm{fix}_{\Theta }(H,w)$ is given by
\[
V_{\Theta }\left( H,w\right)_x = \frak{l}_{x} \cdot x, \qquad
\mbox{where} \qquad \frak{l}_{x} = \frak{l}^+_{x} \oplus
\frak{l}^-_{x}.
\]
It follows that the map
\[
X \in \frak{l}_{x}\mapsto X \cdot x \in V_{\Theta }\left(
H,w\right)_x
\]
is a linear isomorphism, for each $x \in \mathrm{fix}_{\Theta
}(H,w)$.

Now we define the linearization map by
\begin{equation}  \label{eqlinearizacao}
\psi :V_{\Theta }(H,w)\to {\mathbb F}_{\Theta }, \qquad \psi
\left( X \cdot x \right) =\exp (X)x,
\end{equation}
where $X \in \frak{l}_{x}$ and $x \in {\rm fix}_\T(H,w)$. By
Theorem 3.4 of \cite{pss-conley}, we have the following result.

\begin{teorema}\label{teolineariz}
The map $\psi : V_{\Theta }(H,w)\to {\mathbb F}_{\Theta }$ takes
the null section $V_{\Theta }(H,w)_0$ onto ${\rm fix}_\T(H,w)$ and
satisfies:

\begin{itemize}
\item[i)] Its restriction to some neighborhood ${\cal N}$ of
$V_{\Theta }(H,w)_0$ in $V_{\Theta }(H,w)$ is a diffeomorphism
over a neighborhood $N$ of ${\rm fix}_\T(H,w)$ in $\F_\T$.

\item[ii)] Its restrictions to $V^\pm_{\Theta }(H,w)$ are
diffeomorphisms, respectively, onto ${\rm un}_\T(H,w)$ and ${\rm
st}_\T(H,w)$.
\end{itemize}
\end{teorema}

The above map $\psi : V_{\Theta }(H,w)\to {\mathbb F}_{\Theta }$
is a conjugation of the flow $g^t$ on a neighborhood of the
attractor component $\mathrm{fix}_{\Theta }\left( H,w\right)$ to
the linear flow
\[
g^t(X \cdot x) = g^t X \cdot g^t x
\]
on a neighborhood of the null section of $V_{\Theta }(H,w)$ if and
only if $V_{\Theta }(H,w)$ is invariant and $\psi$ is equivariant
by the flow $g^t$. A sufficient condition is that the map $x
\mapsto \mathfrak{l}_x$ be equivariant by $g^t$, i.e, that
$g^t\mathfrak{l}_{x} = \mathfrak{l}_{g^tx}$. For the attractor
component $\mathrm{fix}_{\Theta }\left( H,1\right)$, by Corollary
3.6 of \cite{pss-conley}, this happens whenever either $\Theta
\subset \Sigma(H)$ or $\Sigma (H)\subset \Theta$, where
$\Sigma(H)$ is the annihilator of $H$ in the simple roots
$\Sigma$. For the other Morse components, by Proposition 5.5 of
\cite{pfs-jordan}, it is sufficient that $g^t$ be conformal, i.e.,
its unipotent component be trivial.

\section{Lyapunov and metric spectra}

Let $\phi^t$ be a discrete or continuous-time flow defined on a
compact metric space $(X,d)$. If $M \subset X$ is an isolated
compact $\phi^t$-invariant set, its stable set is given by
\[
\mathrm{st}(M) = \{ x \in X : \omega(x) \subset M \}.
\]
Let us assume that $M$ is not a repeller, which is equivalent to
that $\mathrm{st}(M)$ is not empty. We say that $x \in
\mathrm{st}(M)$ is \emph{metric regular} if there exists the
following limit
\[
\lambda(\phi^t,x) = \lim_{t \to \infty} \frac{1}{t} \log
d(\phi^t(x),M),
\]
called the \emph{metric exponent of $\phi^t$ starting at $x$},
measuring the exponential rate of approximation of $\phi^t(x)$ to
$M$. The \emph{metric spectrum of $\phi^t$ relative to $M$} is
defined by
\[
\Lambda(\phi^t,M) = \{\lambda(\phi^t,x) : x \mbox{ is metric
regular}\}.
\]
The following result shows that $\Lambda(\phi^t,M)$ is an
invariant of the conjugation classes by bi-Lipschitz maps.

\begin{proposicao}
Let $(\overline{X},\overline{d})$ be a metric space and $\psi : X
\to \overline{X}$ be a bi-Lipschitz map. If $\overline{\phi}^t =
\psi\phi^t\psi^{-1}$, $\overline{x} = \psi(x)$ and $\overline{M} =
\psi(M)$, then
\[
\lambda(\phi^t,x) = \lambda(\overline{\phi}^t,\overline{x})
\]
and
\[
\Lambda(\phi^t,M) = \Lambda(\overline{\phi}^t,\overline{M}).
\]
In particular, $\Lambda(\phi^t,M)$ is independent of equivalent
metrics.
\end{proposicao}
\begin{proof}
Since $\psi$ is a bi-Lipschitz map, there exist positive constants
$b, c \in \mathbb{R}$ such that
\[
b d(x,y) \leq \overline{d}(\psi(x),\psi(y)) \leq c d(x,y),
\]
for every $x,y \in X$. We have that
\[
\lambda(\overline{\phi}^t,\psi(x)) = \lim_{t \to \infty}
\frac{1}{t} \log \overline{d}(\overline{\phi}^t(\psi(x)),\psi(M))
= \lim_{t \to \infty} \frac{1}{t}
\log\overline{d}(\psi(\phi^t(x)),\psi(M))
\]
and thus
\[
\lim_{t \to \infty} \frac{1}{t} \log b d(\phi^t(x),M) \leq
\lambda(\psi(x),\overline{\phi}^t) \leq \lim_{t \to \infty}
\frac{1}{t} \log c d(\phi^t(x),M)
\]
showing that $\lambda(\phi^t,x) =
\lambda(\overline{\phi}^t,\psi(x))$, since
\[
\lambda(\phi^t,x) = \lim_{t \to \infty} \frac{1}{t} \log b
d(\phi^t(x),M) = \lim_{t \to \infty} \frac{1}{t} \log c
d(\phi^t(x),M).
\]
The last assertion follows, since two metrics $d$ and
$\overline{d}$ are equivalent if and only if the identity map from
$(X, d)$ to $(X, \overline{d})$ is a bi-Lipschitz map.
\end{proof}

Since diffeomorphisms between Riemannian manifolds are locally
bi-Lipschitz maps, we have the following result.

\begin{corolario}\label{corlipschitz}
Let $\psi : X \to \overline{X}$ be a diffeomorphism between the
Riemannian manifolds $(X,d)$ and $(\overline{X},\overline{d})$. If
$\overline{\phi}^t = \psi\phi^t\psi^{-1}$, $\overline{x} =
\psi(x)$ and $\overline{M} = \psi(M)$, then
\[
\lambda(\phi^t,x) = \lambda(\overline{\phi}^t,\overline{x})
\]
and
\[
\Lambda(\phi^t,M) = \Lambda(\overline{\phi}^t,\overline{M}).
\]
In particular, $\Lambda(M,\phi^t)$ is independent of Riemannian
metrics.
\end{corolario}
\begin{proof}
Since $X$ is locally compact and $M$ is a compact subset, there
exists an compact neighborhood $B$ of $M$. Since $\psi$ is a
diffeomorphism, it is a locally bi-Lipschitz map and its
restriction to $B$ is a bi-Lipschitz map onto its image.
\end{proof}

Now we establish the connection between the Lyapunov and the
metric spectra. We recall that for a linear flow $\Phi^t$ of a
normed vector bundle $(V,|\cdot|)$, we say that $v \in V$ is
\emph{Lyapunov regular} if there exists the following limit
\[
\lambda_L(\Phi^t,v) = \lim_{t \to \infty} \frac{1}{t} \log |\Phi^t
v|,
\]
called the \emph{Lyapunov exponent of $\Phi^t$ starting at $v$},
measuring the exponential rate of growth of $|\Phi^t v|$. The
\emph{Lyapunov spectrum of $\Phi^t$} is defined by
\[
\Lambda_L(\Phi^t,V) = \{\lambda_L(\Phi^t,v) : v \mbox{ is Lyapunov
regular}\}.
\]
The \emph{stable Lyapunov spectrum of $\Phi^t$} is defined by
\[
\Lambda_L(\Phi^t,S) = \{\lambda_L(\Phi^t,v) : v \in S \mbox{ and
}v \mbox{ is Lyapunov regular}\}.
\]
where $S$ is the stable set of the null section of $V$.

\begin{proposicao}\label{proplyapunovmetric}
Let $(X,d)$ be a Riemannian manifold and $\psi : N \to A$ be a
diffeomorphism between some open neighborhood $N$ of the null
section $Z$ of the normal bundle $V$ of $M$ and some open
neighborhood $A$ of $M$. If $M = \psi(Z)$ and
$\psi^{-1}\phi^t\psi$ is the restriction to $N$ of a linear flow
$\Phi^t$ of $V$, then
\[
\lambda(\phi^t,\psi(v)) = \lambda_L(\Phi^t,v),
\]
for all $v \in S$, Lyapunov regular, and
\[
\Lambda(\phi^t,M) = \Lambda_L(\Phi^t,S).
\]
Furthermore, we have that $\Lambda(\phi^t,M)$ is not empty.
\end{proposicao}
\begin{proof}
There exist $s > 0$ such that $\Phi^s v \in N$ and
$\phi^s(\psi(v)) \in A$. Since $\lambda(\Phi^t,v) =
\lambda(\Phi^t,\Phi^s v)$ and $\lambda(\phi^t,\psi(v)) =
\lambda(\phi^t,\phi^s(\psi(v)))$, we can assume that $v \in N$ and
$\psi(v) \in A$. By Corollary \ref{corlipschitz}, we have that
\[
\lambda(\Phi^t, v) = \lambda(\Phi^t|_N, v) = \lambda(\phi^t|_A,
\psi(v)) = \lambda(\phi^t, \psi(v)).
\]
Thus it remains to prove that $\lambda(\Phi^t, v) =
\lambda_L(\Phi^t, v)$, for all $v \in S$, metric regular. Let
$\overline{d}$ be the Sasaki metric on the normal bundle $V$ (see
Section 4.2 of \cite{borisenko}). We can choose $N$ such that
$\overline{d}(v,Z) = |v|$, for all $v \in N$. Thus, for all $v \in
S$, we have that $v$ is metric regular if and only if $v$ is
Lyapunov regular and, in this case, $\lambda(\Phi^t,v) =
\lambda_L(\Phi^t,v)$.

For the last assertion, we first observe that $Z$ is a
$\Phi^t$-invariant and isolated, since $M$ is a $\phi^t$-invariant
and isolated. Thus the restriction of $S$ over some chain
transitive component of the flow induced by $\Phi^t$ on the base
$M$ is a subbundle (see Theorem 2.13 of \cite{sz}). Thus we can
apply Oselec theorem (see \cite{oseledec}) to the restriction of
$\Phi^t$ to this subbundle, showing that $\Lambda_L(\Phi^t,S)$ is
not empty.
\end{proof}

The previous proposition and Theorem 1 of \cite{ps} imply the
following result.

\begin{corolario}
If $M$ is normally hyperbolic, then $\Lambda(\phi^t,M)$ is not
empty.
\end{corolario}

\section{Flag spectrum}

In this section, we determine explicitly the metric spectrum of a
linear induced flow $g^t$ relative to a Morse component ${\rm
fix}_\T(H,w)$. First we need to introduce some suitable
constructions, which are related to the constructions presented in
the Section \ref{preliminaries}. We denote
\[
\Lambda_\T(H,w) = \{\alpha(H) : \g_\alpha \subset \l_{wb_\T}\},
\]
where
\[
\l_{wb_\T} = (\n^+_H \oplus \n^-_H) \cap w\n^-_\T.
\]
Now writing
\begin{equation}\label{eqLambda}
\Lambda_\T(H,w) = \{ \lambda_1 > \cdots > \lambda_{n_\T^w}\},
\end{equation}
for each $i \in \{1,\ldots,n_\T^w\}$, we denote
\[
\b_i = \{X : \ad(H)X = \lambda_i X \}
\]
and define the family of subspaces given by
\begin{equation}\label{eqlix}
\l^i_x = \l_x \cap \sum \{ \b_j : j \geq i \}.
\end{equation}

\begin{proposicao}\label{proplix}
For each $i \in \{1,\ldots,n_\T^w\}$, we have that
\[
\l^i_x = k \l^i_{wb_\T}
\]
for $x = k w b_\T$, where $k \in K_H$. Furthermore, if $H$ is the
hyperbolic type of $g^t$ and
\begin{enumerate}[$(i)$]
\item $w$ and $\T$ are arbitrary, with $g^t$ conformal or

\item $w=1$ and $\Theta \subset \Sigma(H)$ or $\Sigma (H)\subset
\Theta$, with $g^t$ arbitrary,
\end{enumerate}
then
\[
g^t \l^i_x = \l^i_{g^tx}.
\]
\end{proposicao}
\begin{proof}
First we note that the eigenspaces of $\ad(H)$ are invariant by
the centralizer $G_H$ of $H$ in $G$, since the elements of $G_H$
commute with $\ad(H)$. The first claim follows, since $\l_x = k
\l_{wb_\T}$, for $x = k w b_\T$, where $k \in K_H \subset G_H$. By
Corollary 3.6 and 3.9 of \cite{pss-conley} and Proposition 5.5 of
\cite{pfs-jordan}, if $(i)$ or $(ii)$ are verified, then we that
$g^t\mathfrak{l}_{x} = \mathfrak{l}_{g^tx}$. Thus the second claim
follows, since $g^t \in G_H$.
\end{proof}

Following the proof of Proposition 3.2 of \cite{pss-conley}, for
each $i \in \{1,\ldots,n_\T^w\}$, we have that
\[
V^i_\T(H,w) = \bigcup \{ \frak{l}^i_{x}\cdot x:\, x \in
\mathrm{fix}_{\Theta }(H,w) \}
\]
is a differentiable subbundle of the tangent bundle of
$V_\T(H,w)$. The norm in $V_{\Theta }(H,w)$ is the restriction of
the norm in the tangent bundle of $\F_\T$ induced by the
Riemannian metric introduced in Section \ref{preliminaries}. We
need to prove an elementary fact about the norm $|\cdot|$ in
$V_{\Theta }(H,w)$.

\begin{lema}\label{lemaspherical}
If $v = X \cdot x \in V_{\Theta }(H,w)$, where $X \in
\frak{l}_{x}$ and $x \in {\rm fix}_\T(H,w)$, then $|v| = |X|$.
\end{lema}
\begin{proof}
We have that $X \in \frak{l}_{x} = k\frak{l}_{wb_\T}$ and that $x
= kwb_\T$, for some $k \in K_H$. Thus $Y \in k^{-1}X \in
\frak{l}_{wb_\T}$ and
\[
v = k(Y \cdot wb_\T) = kw(w^{-1} \cdot b_\T).
\]
Thus, by the definition of $|\cdot|$, we have that
\[
|v| = |w^{-1}Y| = |Y| = |k^{-1}X| = |X|,
\]
where we use that $K$ acts in $\g$ by
$\prod{\cdot,\cdot}$-isometries.
\end{proof}

Now we prove a strong version of Oseledec theorem (see
\cite{oseledec}), determining explicitly the Oseledec
decomposition of $V_\T(H,w)$ relative to the flow $g^t$.

\begin{teorema}\label{teospherical}
Let $H$ be the hyperbolic type of $g^t$. For each
\begin{enumerate}[$(i)$]
\item arbitrary $w$ and $\T$, when $g^t$ is conformal or

\item $w=1$ and $\Theta \subset \Sigma(H)$ or $\Sigma (H)\subset
\Theta$, when $g^t$ is arbitrary,
\end{enumerate}
each $i \in \{1,\ldots,n_\T^w\}$ and each $v \in V^i_\T(H,w) -
V^{i-1}_\T(H,w)$, we have that
\[
\lambda_L(g^t,v) = \lambda_i.
\]
In particular, every point of $V_\T(H,w)$ is Lyapunov regular and
\[
\Lambda_L(g^t,V_\T(H,w)) = \Lambda_\T(H,w).
\]
\end{teorema}
\begin{proof}
Using Proposition \ref{proplix} and the definition of
$V^i_\T(H,w)$, we have that $v = k(X \cdot wb_\T)$, for some $k
\in K_H$ and some $X \in \l_{wb_\T}^i - \l_{wb_\T}^{i-1}$. Writing
$X = \sum_\alpha X_\alpha$, with $X_\alpha \in \g_\alpha$, by the
equations (\ref{eqLambda}) and (\ref{eqlix}), we have that
\[
\lambda_i = \max \{\alpha(H) : X_\alpha \neq 0\}.
\]
On the other hand, by Lemma \ref{lemaspherical}, we have that
\begin{equation}\label{eq1}
|g^t v| = |g^t kX \cdot g^t kwb_\T| = |g^t kX|.
\end{equation}
Let $g^t = e^t u^t h^t$ be the multiplicative Jordan decomposition
of the flow $g^t$. We have that $e^t$ acts on $\g$ by
$\prod{\cdot,\cdot}$-isometries, that $u^t = \exp(tN)$ with
$\ad(N)$ a nilpotent operator and that $h^t$ commutes with $K_H$.
Thus we have that
\[
|g^t kX| = |e^t u^t h^t kX| = |u^t k h^t X| = |{\rm e}^{t\ad(N)} k
h^t X|,
\]
since $\exp(tN)Z = {\rm e}^{t\ad(N)}Z$. Hence
\begin{equation}\label{eq2}
|g^t kX| = |{\rm e}^{t\ad(N)} k h^t X| = |k^{-1}{\rm e}^{t\ad(N)}
k h^t X| = |{\rm e}^{t\mathcal{N}} h^t X|,
\end{equation}
where $\mathcal{N} = k^{-1}\ad(N)k$ is also a nilpotent operator.

By equations (\ref{eq1}) and (\ref{eq2}), denoting $\lambda =
\lambda_i$, it is sufficient to show that
\begin{equation}\label{eq3}
\lambda = \lim_{t \to \infty} \frac{1}{t} \log |{\rm
e}^{t\mathcal{N}} h^t X|.
\end{equation}
Since $h^t X_\alpha = {\rm e}^{\alpha(H)t} X_\alpha$, we have that
\begin{equation}\label{eq4}
{\rm e}^{t\mathcal{N}} h^t X = \sum_\alpha {\rm e}^{\alpha(H)t}
{\rm e}^{t\mathcal{N}} X_\alpha = \sum_\alpha {\rm e}^{\alpha(H)t}
\sum_{n} \frac{t^n}{n!} \mathcal{N}^n X_\alpha.
\end{equation}
Denote
\[
Y = \sum_{\alpha(H) = \lambda} X_\alpha
\]
and
\[
l = \max \{n : \mathcal{N}^n Y \neq 0 \} \leq m,
\]
where $m \in \mathbb{N}$ is such that $\mathcal{N}^{m+1} = 0$. In
order to get equation (\ref{eq3}), we just need to prove that
\begin{equation}\label{eq5}
\lim_{t \to \infty} \frac{|{\rm e}^{t\mathcal{N}} h^t
X|}{\frac{{\rm e}^{\lambda t}t^l}{l!} |\mathcal{N}^lY|} = 1,
\end{equation}
since
\[
\lim_{t \to \infty} \frac{1}{t} \log \left(\frac{{\rm e}^{\lambda
t}t^l}{l!} |\mathcal{N}^lY|\right) =  \lambda.
\]
Using equation (\ref{eq4}), we have that
\begin{equation}\label{eq6}
\frac{|{\rm e}^{t\mathcal{N}} h^t X|}{\frac{{\rm e}^{\lambda
t}t^l}{l!} |\mathcal{N}^lY|} = \frac{\left|\sum_\alpha {\rm
e}^{(\alpha(H) - \lambda)t} \sum_{n} \frac{t^{n - l}}{n!}
\mathcal{N}^n X_\alpha\right|}{\frac{1}{l!} |\mathcal{N}^lY|} =
\frac{|U_t + V_t|}{ |\frac{1}{l!}\mathcal{N}^lY|}
\end{equation}
where
\[
U_t = \sum_{\alpha(H) = \lambda} {\rm e}^{(\alpha(H) - \lambda)t}
\sum_{n=0}^m \frac{t^{n - l}}{n!} \mathcal{N}^n X_\alpha =
\sum_{n=0}^l \frac{t^{n - l}}{n!} \mathcal{N}^n Y
\]
and
\[
V_t = \sum_{\alpha(H) < \lambda} {\rm e}^{(\alpha(H) - \lambda)t}
\sum_{n=0}^m \frac{t^{n - l}}{n!} \mathcal{N}^n X_\alpha.
\]
It is immediate that
\[
\lim_{t \to \infty} U_t = \frac{1}{l!}\mathcal{N}^lY \qquad
\mbox{and} \qquad \lim_{t \to \infty} V_t = 0,
\]
and thus equation (\ref{eq6}) implies the equation (\ref{eq5}).

The last assertion follows, since we have that $V_\T(H,w)$ is
given by the disjoint union of $V^i_\T(H,w) - V^{i-1}_\T(H,w)$,
for $i \in \{1,\ldots,n_\T^w\}$.
\end{proof}

We denote
\[
\Lambda_\T^-(H,w) = \Lambda_\T(H,w) \cap \mathbb{R}^-
\]
and, for each $\lambda_i \in \Lambda_\T^-(H,w)$, we define
\[
{\rm st}_\T^i(H,w) = \psi(V^i_\T(H,w) - V^{i-1}_\T(H,w)),
\]
where $\psi: V_\T(H,w) \to \F_\T$ is the map presented in Theorem
\ref{teolineariz}. It is immediate that the stable set of ${\rm
fix}_\T(H,w)$ relative to the flow $g^t$ is given by the following
disjoint union
\[
{\rm st}_\T(H,w) = \bigcup \{{\rm st}_\T^i(H,w) : \lambda_i \in
\Lambda_\T^-(H,w)\}.
\]
The following result determines explicitly the metric spectrum of
$g^t$ relative to the Morse components.

\begin{corolario}\label{corspherical}
Let $H$ be the hyperbolic type of $g^t$. For each
\begin{enumerate}[$(i)$]
\item arbitrary $w$ and $\T$, when $g^t$ is conformal or

\item $w=1$ and $\Theta \subset \Sigma(H)$ or $\Sigma (H)\subset
\Theta$, when $g^t$ is arbitrary
\end{enumerate}
and each $x \in {\rm st}_\T^i(H,w)$, we have that
\[
\lambda(g^t,x) = \lambda_i.
\]
In particular, every point of ${\rm st}_\T(H,w)$ is metric regular
and
\[
\Lambda(g^t,{\rm fix}_\T(H,w)) = \Lambda_\T^-(H,w).
\]
\end{corolario}
\begin{proof}
Under the above hypothesis, we have that $V_{\Theta }(H,w)$ is
invariant and $\psi : V_{\Theta }(H,w)\to {\mathbb F}_{\Theta }$
is equivariant by the flow $g^t$. Hence the corollary is a
immediate consequence of Propositions \ref{propmorseflag} and
\ref{proplyapunovmetric} and Theorems \ref{teolineariz} and
\ref{teospherical}.
\end{proof}

The previous result allows us the following definition. The
\emph{flag spectrum} of an arbitrary linear induced flow $g^t$ is
given by $\Lambda^-(H,1)$, its metric spectrum relative to the
unique attractor component ${\rm fix}(H,1)$ in the maximal flag
manifold $\F$, where $\T = \emptyset$.

\end{document}